\documentclass[11pt]{article}
\usepackage{amsmath, amssymb, amsthm}
\usepackage{enumitem}                                 
\usepackage[margin=1.0in]{geometry}

\AtEndDocument{\bigskip{\footnotesize%
  \textsc{Department of Mathematics, the Pennsylvania State University, University
Park, PA 16802, USA} \par  
  \textit{E-mail address} \texttt{ amir.algom1@gmail.com
} 
}}

\newtheorem{theorem}{Theorem}[section]

\newtheorem{Counter-example}[theorem]{Counter example}
\newtheorem{Claim}[theorem]{Claim}

\newtheorem{Proposition}[theorem]{Proposition}

\newtheorem*{theorem*}{Theorem}

\newcommand{\supp}{\text{supp}}

\newcommand\blfootnote[1]{%
  \begingroup
  \renewcommand\thefootnote{}\footnote{#1}%
  \addtocounter{footnote}{-1}%
  \endgroup
}

\usepackage{amsthm}

\newtheorem{innercustomgeneric}{\customgenericname}
\providecommand{\customgenericname}{}
\newcommand{\newcustomtheorem}[2]{%
  \newenvironment{#1}[1]
  {%
   \renewcommand\customgenericname{#2}%
   \renewcommand\theinnercustomgeneric{##1}%
   \innercustomgeneric
  }
  {\endinnercustomgeneric}
}

\newcustomtheorem{customthm}{Theorem}
\newcustomtheorem{customlemma}{Lemma}

\makeatother

\title{Actions of diagonal endomorphisms on conformally invariant measures on the 2-torus}

\author{Amir Algom }

\date{}

\begin{document}
\maketitle

\begin{abstract}
Let\blfootnote{Supported by ERC grant 306494 and ISF grant 1702/17.}  $\nu$ be a probability measure  that is  ergodic under the endomorphism $(\times p, \times p)$ of the torus $\mathbb{T}^2$, such that  $\dim \pi \mu < \dim \mu$ for some non-principal projection $\pi$. We show that, if both $m\neq n$ are independent of $p$,  the $(\times m, \times n)$ orbits of $\nu$ typical points will equidistribute towards the Lebesgue measure. If $m>p$ then typically the $(\times m, \times p)$ orbits  will equidistribute towards the product of the Lebesgue measure with the marginal of $\mu$ on the $y$-axis. We also prove results in the same spirit for certain self similar measures $\nu$. These are  higher dimensional analogues of results due (among others) to Host \cite{Host1995normal}, Lindenstrauss \cite{Elon2001host}, and Hochman-Shmerkin \cite{hochmanshmerkin2015}.   
\end{abstract}

\section{Introduction}
Let $p$ be an integer greater or equal to $2$. Let $T_p$ be the $p$-fold map of the unit interval,  $$T_p (x) = p\cdot x \mod 1.$$
We  say that an integer $m>1$ is independent of $p$ if $\frac{\log p}{\log m} \notin \mathbb{Q}$, and write $m \not \sim p$ to indicate that $m$ and $p$ are independent. The purpose of this paper is to study higher dimensional analogues of the following  Theorem.

\begin{customthm}{$\star$}\label{Theorem Host template}
Let $\mu$ be a $T_p$ invariant ergodic measure with positive entropy, and let  $m\not \sim p$. Then $\mu$ almost every $x$ is normal in base $m$, that is, the sequence $\lbrace T_m ^k x \rbrace_{k \in \mathbb{Z}_+}$ equidistributes for $\lambda$, the Lebesgue measure on $[0,1]$.
\end{customthm}

Theorem \ref{Theorem Host template} was originally proved around 1960 for some  specific Cantor-Lebesgue measures $\mu$ by Cassels  \cite{Cassels1960normal} and Schmidt \cite{Schmidt1960normal}.  This was later generalized by Feldman and Smorodinsky \cite{Feldman1992normal} to all non-degenerate Cantor-Lebesgue measures (in fact, weakly Bernoulli). In 1995 Host \cite{Host1995normal} proved Theorem \ref{Theorem Host template} under the assumption that $m$ and $p$ are co-prime. Host's Theorem can be shown to imply Rudolph's Theorem \cite{Rudolph1990dis} which, together with Johnson's work \cite{Johnson1992dis}, proved the positive entropy case of Furstenberg's $\times 2, \times 3$ Conjecture \cite{furstenberg1967disjointness}. In 2001, Lindenstrauss \cite{Elon2001host} proved Theorem \ref{Theorem Host template} under the  assumption that $p$ does not divide any power of $m$ (which is weaker than Host's assumption). Finally, in 2015 Hochman and Shmerkin \cite{hochmanshmerkin2015} proved Theorem \ref{Theorem Host template} in its present form, i.e. whenever $p\not \sim m$. See (\cite{hochmanshmerkin2015}, Section 1.1) for some more discussion on the background of Theorem \ref{Theorem Host template}. 

We next survey some known higher dimensional analogues of Theorem \ref{Theorem Host template}. Let $S,T$ be endomorphisms of $\mathbb{T}^d$, where $\mathbb{T}=\mathbb{R}/\mathbb{Z}$, and let $\nu$ be an $S$ invariant and ergodic probability measure with positive (lower) Hausdorff dimension. Recall that for a probability measure $\mu$  its (lower) Hausdorff dimension is defined  by
\begin{equation} \label{Eq Hausdorff dimension}
\dim \mu = \inf \lbrace \dim_H A:\quad \mu(A)>0 \rbrace 
\end{equation}
where $\dim_H A$ is the Hausdorff dimension of the set $A$ (see e.g. \cite{falconer1986geometry}). The statistical behavior of the orbits $\lbrace T^k x \rbrace_{k\in \mathbb{Z}_+}$ for $\nu$ typical $x$ was  studied by several authors, under various assumptions:
\begin{enumerate}
\item Meiri and Peres \cite{Meiri1998Peres} worked with two diagonal endomorphisms $S$ and $T$,  requiring that the corresponding diagonal entries $S_{i,i}$ and $T_{i,i}$ be larger than $1$ and co-prime. They proved that the projections of theses orbits onto some non-trivial sub-torus  equidistribute towards the Lebesgue measure (on this sub-torus).

\item  Host \cite{Host2000high} showed that these orbits will almost surely equidistribute towards the Lebesgue measure, under the following  assumptions: all eigenvalues of $S$ have modulus $>1$,  for every $k$ the characteristic polynomial of  $T^k$ is irreducible over $\mathbb{Q}$, and $\det (S)$ and $\det(T)$ are co-prime.

\item Recently, in \cite{algom2019simultaneous}, we proved a simultaneous version of Theorem \ref{Theorem Host template}. Namely, we  considered the case when \begin{equation*}
S = \begin{pmatrix}
p & 0 \\
0& p 
\end{pmatrix} = T_p \times T_p, \quad 
T = \begin{pmatrix}
m & 0 \\
0& p
\end{pmatrix} = T_m \times T_p
\end{equation*}
and the measure $\nu=\Delta \mu$, where $\Delta:\mathbb{T}\rightarrow \mathbb{T}^2$ is the map $\Delta(x)=(x,x)$, for a probability measure $\mu$  as in Theorem \ref{Theorem Host template}. We showed that if  $m\not \sim  p$ and $m>p$ then for  $\nu$ typical $(x,x)$ their $T$-orbits will equidistribute towards $\lambda \times \mu$. We also showed that if   $m>n>p$ and $n\not \sim p$, then for $\nu$ typical $(x,x)$ their $T_m \times T_n$ orbits will equidistribute towards $\lambda \times \lambda$. In fact, similar  results hold true when (affinely)  embedding $\mu$ into any line through $[0,1]^2$ not parallel to the major axes (\cite{algom2019simultaneous}, Theorem 6.1).

\end{enumerate}

Let $T$ be an endomorphism of $\mathbb{T}^2$. Following \cite{hochmanshmerkin2015}, we say that a measure $\mu \in \mathcal{P}([0,1]^2)$  is pointwise generic under $T$ for a  measure $\rho$ if  $\mu$ almost every $z$ equidistributes for $\rho$ under $T$, that is,
\begin{equation*}
\frac{1}{N} \sum_{k=0} ^{N-1} f(T^k z) \rightarrow \int f(x) d\rho(x),\quad \forall f\in C(\mathbb{T}^2),
\end{equation*}
where for a compact metric space $X$ we let $\mathcal{P}(X)$ be the space of Borel probability measures on $X$. Also, let $P_1(x,y)=x$ and $P_2(x,y)=y$ be the principal projections.

The main result of this paper is a generalization of our results from \cite{algom2019simultaneous} (stated in (3) above) to  $T_p \times T_p$ invariant measures, that are not necessarily supported on lines.

\begin{theorem} \label{Theorem inv}
Let $m>p\geq 2$ be integers such that $m\not \sim p$. Let $\mu\in \mathcal{P}([0,1]^2)$ be an ergodic $T_p \times T_p$ invariant measure satisfying the following condition:
\begin{equation} \label{Equation condition}
\text{ There exists a linear projection } \pi:\mathbb{R}^2 \rightarrow \mathbb{R} \text { such that } \dim \pi \mu < \dim \mu, \text{ and } \pi \not \in \lbrace P_1,P_2 \rbrace.
\end{equation}

Then:
\begin{enumerate}
\item The measure $\mu$ is pointwise generic for $\lambda \times P_2 \mu$ under the map $T_m \times T_p$.

\item If $n$ is another integer such that $n \not \sim p$ and $m>n>p$ then $\mu$ is pointwise generic for $\lambda \times \lambda$ under the map $T_m \times T_n$.
\end{enumerate}
\end{theorem}

Note that if an ergodic $T_p \times T_p$ invariant  measure $\mu$ satisfies at least one of the following conditions, then it also satisfies  \eqref{Equation condition}:
\begin{enumerate}
\item $\dim \mu >1$.

\item $\dim \mu >0$ and $\mu$ is supported on a line not parallel to either one of the major axes.

\item $\mu$ is the distribution of a random sum $\sum_{k=1} ^\infty \frac{X_k}{p^k}$ where the $X_k$'s are non-degenerate random variables that form an IID sequence and take values in $\lbrace0,...,p-1\rbrace\times \lbrace 0,...,p-1\rbrace$, and $\mu$   is not supported on a line  parallel to either one of the major axes.  
\end{enumerate}
We leave the (standard) proof to the interested reader. From the second example  above we see that the main result of \cite{algom2019simultaneous} follows as a special case of Theorem \ref{Theorem inv}. It is an interesting question to determine weather or not every ergodic $T_p \times T_p$ invariant measure with positive dimension admits a  projection $\pi:\mathbb{R}^2 \rightarrow \mathbb{R}$  such that  $\dim \pi \mu < \dim \mu$  and  $\pi \not \in \lbrace P_1,P_2 \rbrace$.   We remark that for every $p\geq 2$ one may construct a  $T_{p^2} \times T_p$ invariant measure  as in the third example above, that does not admit such a projection. For instance, consider the measure $\mu$ that is the distribution of the random sum $\sum_{k=1} ^\infty (\frac{X_k}{4^k}, \frac{Y_k}{2^k})$, where the pairs $(X_k,Y_k)$ form an IID sequence with
\begin{equation*}
\mathbb{P}( (X_1,Y_1) = (0,0) ) = \mathbb{P}( (X_1,Y_1) = (1,1) ) = \frac{1}{2}
\end{equation*}
Then $\min \lbrace \dim P_1 \mu, \dim P_2 \mu \rbrace =\frac{1}{2}$ so both $P_1 \mu$ and $P_2 \mu$ have no atoms, $\mu$ is $T_{4} \times T_2$ invariant, $\dim \mu =1$ (see either \cite{bedford1984crinkly} or \cite{mcmullen1984hausdorff}), and $\dim \pi \mu =1$ for every $\pi \neq P_1$.  The last assertion may be deduced from the results in \cite{Fraser2015Ferguson}.

Next we discuss an analogue of Theorem \ref{Theorem inv} for self similar measures: consider a self similar IFS (Iterated Function System) on $\mathbb{R}^2$ of the form
\begin{equation} \label{Eq ss IFS}
\Phi = \lbrace \phi_k(x)=r_k \cdot O_k\cdot x+t_k \rbrace_{k=1} ^l
\end{equation}
where $0<r_k<1$, $O_k$ are $2\times 2$  rotation matrices for every $k$, and $t_k\in \mathbb{R}^2$. We shall always assume $\Phi$ satisfies the strong separation  condition, which we abbreviate by SSC (see Section \ref{Section dimension} for its definition). Recall that a self similar measure is a probability measure of the form
\begin{equation} \label{Eq ss mea}
\mu = \sum_{k=1} ^l p_k \cdot \phi_k \mu
\end{equation}
where $p=(p_1,...,p_l)$ is a non-trivial probability vector (its existence and uniqueness are due to Hutchinson \cite{hutchinson1981fractals}). We will always assume that all the entries of $p$ are positive. Also, let $D\Phi := \lbrace O_k \rbrace_{k=1} ^l \subset SO(\mathbb{R}^2)$ be the finite set of rotation matrices associated with $\Phi$, and define
\begin{equation*}
k_\Phi := \min \lbrace k\in \mathbb{N}:\quad  O_1 ^k = O_2 ^{k}, \text{ for all } O_1, O_2 \in D\Phi \rbrace, 
\end{equation*}
where we let $\min \emptyset = \infty$. If $k_\Phi<\infty$, we define the angle
\begin{equation*}
\gamma_\Phi := \text{ The unique } \gamma \in [0,2\pi) \text{ such that for any } O\in D\Phi, O^{k_\Phi} \text{ is the rotation by } \gamma.
\end{equation*}

Finally, we define $G_\Phi$ to be the group generated by the orthogonal parts of the similarities in $\Phi$, that is, $G_\Phi = < D\Phi>$ is a subgroup of $SO(\mathbb{R}^2)$.

\begin{theorem} \label{Theorem ss}
Let $m>n>1$ be integers and $0<r<1$ be such that both $m,n\not \sim r$.  Let $\mu \in \mathcal{P}([0,1]^2)$ be a self similar measure with respect to an IFS $\Phi$ such that:
\begin{enumerate}
\item $\Phi$ satisfies the SSC, has a uniform contraction ratio $r$, $\supp(\mu)$ does not lie on a vertical or horizontal line, and $G_\Phi \leq SO(\mathbb{R}^2)$.

\item Either $|G_\Phi|<\infty$ or $\dim \mu >1$.

\item If $k_\Phi<\infty$ then for every integer $q\neq 0$ both $\frac{q}{\log m},\frac{q}{\log n} \not \not \in  \mathbb{Z}  \cdot  \frac{\gamma_\Phi}{\log r \cdot k_\Phi}$.

\end{enumerate}
Then $\mu$ is pointwise generic for $\lambda \times \lambda$ under the map $T_m \times T_n$.
\end{theorem}

Theorem \ref{Theorem ss} is a higher dimensional analogue of a Theorem of Hochman and Shmerkin (\cite{hochmanshmerkin2015}, Theorem 1.4) about similar results for a wide class of one dimensional IFS's.  Our assumptions about the independence of the contraction ratio $r$ from $m,n$ are analogues to theirs. Our other assumptions  (e.g. about the relation of $m,n$ to $\gamma_\Phi$ and $k_0$, and that $D\Phi$ comprises only of rotations) arise as a by-product of our proof.

We dedicate the final part of this introduction to a brief outline of our method. The proof of  Theorem \ref{Theorem ss} relies on finding a projection $\pi:\mathbb{R}^2 \rightarrow \mathbb{R}, \pi \neq P_1,P_2,$ such that:
\begin{enumerate}
\item  The conditional measures $ \mu_{\pi^{-1} (x)}$ almost surely  generate a non-trivial  ergodic fractal distribution  (defined in Section \ref{Section EFD}).

\item The pure point spectrum of this distribution does  not contain non-zero integer multiples of either $\frac{1}{\log m}$ or $\frac{1}{\log n}$.
\end{enumerate}
Condition (3) of Theorem \ref{Theorem ss} arises from the computation of this pure point spectrum. Once such a projection is produced, we appeal to our previous results (\cite{algom2019simultaneous}, Theorem 6.2) about the $T_m \times T_n$ orbits of measures supported on \textit{lines} satisfying these conditions. Now, finding such a projection is hopeless when $|G_\Phi|=\infty$ if we remove the assumption $\dim \mu >1$. Indeed, if $\dim \mu \leq 1$, then by (\cite{hochman2009local}, Theorem 1.6) $\dim \pi \mu = \dim \mu$ for every such projection, so $ \dim \mu_{\pi^{-1} (x)}  =0$ almost surely (by e.g. \cite{algom2019simultaneous}, Lemma 2.2). Such measures are known to generate trivial ergodic fractal distributions \cite{hochman2010dynamics}.

The proof of Theorem \ref{Theorem inv} also relies on finding a suitable disintegration of $\mu$: First, we disintegrate $\mu$ according to the projection $\pi$ as in \eqref{Equation condition}.  Then we will show that typical conditional measures themselves admit a further disintegration such that almost surely the two properties listed above hold. Finally, we apply our previous results about the $T_m \times T_n$ orbits of measures supported on lines to typical conditional measures.

\textbf{Organization} In Section \ref{Section pre} we survey some basic definitions regarding dimension theory of measures and of their scaling sceneries. We also recall some of the results from \cite{algom2019simultaneous} that we shall apply here. We proceed to prove Theorem \ref{Theorem ss}, and then Theorem \ref{Theorem inv}. 

\textbf{Acknowledgments} I am grateful to Mike Hochman, Zhiren Wang, and Federico Rodriguez Hertz for some useful remarks and suggestions.

\section{Preliminaries} \label{Section pre}

 \subsection{Dimensions of measures and of their projections} \label{Section dimension}
Let $\mu \in \mathcal{P}(\mathbb{R}^d)$. For every $x\in \supp(\mu)$  we define the   local (pointwise) dimension of $\mu$ at $x$  as
\begin{equation*}
\dim(\mu,x)=\liminf_{r\rightarrow 0} \frac{\log \mu (B(x,r))}{\log r}
\end{equation*}
where $B(x,r)$ denotes the closed ball or radius $r$ about $x$. The Hausdorff dimension of $\mu$, which we defined in \eqref{Eq Hausdorff dimension}, is equal to
\begin{equation} \label{Eq lower dim}
\dim \mu = \text{ess-inf}_{x\sim \mu} \dim(\mu,x),
\end{equation} 
see e.g. \cite{falconer1997techniques}. If $\dim (\mu,x)$ exists as a limit at almost every point, and is constant almost surely, we shall say that the measure $\mu$ is exact dimensional. 

Now, let $\mu \in \mathcal{P}(\mathbb{R}^2)$ be a probability measure, and let $\pi:\mathbb{R}^2\rightarrow \mathbb{R}$ be a linear projection (a re-parameterization of a projection onto a line). It is a classical question in geometric measure theory to study the dimension of the projected measure $\pi \mu$. Let us concentrate our attention on self similar measures $\mu$ (as in \eqref{Eq ss mea}) with respect to an IFS $\Phi$ (as in \eqref{Eq ss IFS}, where we allow the $O_k$'s to be reflections as well).  Let $X$ be the attractor of $\Phi$, that is, $X\neq \emptyset$ is the unique compact set such that
$$\bigcup_{k=1} ^l \phi_k (X) =X $$
Then $\Phi$ satisfies the SSC (strong separation condition) if the union above is disjoint.  As we have mentioned before, assuming $\Phi$ has strong separation and $|G_\Phi|=\infty$,  Hochman and Shmerkin (\cite{hochman2009local}, Theorem 1.6)  proved that $\dim \pi \mu =\min \lbrace \dim \mu, 1 \rbrace$ for every  projection $\pi$ (later the separation condition was relaxed by Falconer and Jin \cite{Falconer2014Jin}). The situation when $|G_\Phi|<\infty$ is different, as observed by Farkas \cite{Farkas2016proj}:
\begin{theorem} \cite{Farkas2016proj} \label{Theorem farkas}
Let $\mu \in \mathcal{P}(\mathbb{R}^2)$ be  a non degenerate self similar measure with respect to  an IFS $\Phi$ such that $|G_\Phi|<\infty$, and such that $\supp(\mu)$ does not lie on a vertical or horizontal line. Then there exists a non principal projection $\pi \neq P_1,P_2$ such that
\begin{equation*}
\dim \pi \mu < \dim \mu.
\end{equation*}  
\end{theorem}
We remark that while the proof in \cite{Farkas2016proj} deals with  the dimension of projections of  self similar sets, it is not hard to adapt the same proof to work for self similar measures.

\subsection{On the scaling sceneries of measures and ergodic fractal distributions} \label{Section EFD}
In this Section we recall the definition of the scaling scenery of a measure, and related notions. The ideas we introduce here have a long and interesting history, and  we refer the reader to either (\cite{hochmanshmerkin2015}, Section 1.2) or   (\cite{hochman2010dynamics}, Section 1) for some further discussions about them (and also for an exhaustive bibliography). We remark that we  follow the same notation as in \cite{hochmanshmerkin2015}.

Let
\begin{equation} \label{M square}
\mathcal{M}^{\square} = \lbrace \mu \in \mathcal{P}([-1,1]^2):\quad 0\in \supp(\mu) \rbrace.
\end{equation}
For $\mu \in \mathcal{M}^{\square}$ and $t\in \mathbb{R}_+$ we define the scaled measure $S_t \mu \in \mathcal{M}^{\square}$ by
\begin{equation*}
S_t \mu (E) = c \cdot \mu (e^{-t}E\cap [-1, 1]^2),\quad \text{where } c \text{ is a normalizing constant}.
\end{equation*}
For $x\in \supp (\mu)$ we similarly define the translated measure by
\begin{equation*}
\mu^x(E) = c' \cdot  \mu ( (E-x)\cap [-1,1]^2),\quad  \text{where } c' \text{ is a normalizing constant}.
\end{equation*}
The scaling flow is the Borel $\mathbb{R}^+$ flow $S=(S_t)_{t\geq0}$ acting on $\mathcal{M}^{\square} $. The scenery of $\mu$ at $x\in \supp(\mu)$ is the orbit of $\mu^x$ under $S$, that is, the one parameter family of measures $\mu_{x,t}:= S_t(\mu^x)$ for $t\geq0$. Thus, the scenery of the measure at some point $x$ is what one sees as one "zooms" into the measure.

Notice that $\mathcal{P}(\mathcal{M}^{\square} ) \subseteq \mathcal{P}(\mathcal{P}([-1,1]^2))$. As is standard in this context, we shall refer to elements of $\mathcal{P}(\mathcal{P}([-1,1]^2))$ as distributions, and to elements of $ \mathcal{P}(\mathbb{R}^2)$ as measures. A measure $\mu \in \mathcal{P}([0,1]^2)$ generates a distribution $P\in \mathcal{P}(\mathcal{P}([-1,1]^2))$ at $x\in \supp(\mu)$ if the scenery at $x$ equidistributes for $P$ in $\mathcal{P}(\mathcal{P}([-1,1]^2))$, i.e. if
\begin{equation*}
\lim_{T\rightarrow \infty} \frac{1}{T} \int_0 ^T f(\mu_{x,t}) dt = \int f(\nu) dP(\nu),\quad \text{ for all } f\in C( \mathcal{P}([-1,1]^2)).
\end{equation*}
and $\mu$ generates $P$ if it generates $P$ at $\mu$ almost every $x$. If $\mu$ generates $P$, then $P$ is supported on $\mathcal{M}^{\square}$ and is $S$-invariant (\cite{hochman2010dynamics}, Theorem 1.7). We say that $P$ is trivial if it is the distribution supported on $\delta_0 \in \mathcal{M}^{\square}$ - a fixed point of $S$.

The next result says that distributions $P\in \mathcal{P}(\mathcal{P}([0,1]^2))$ that are generated by a given measure $\mu$ have some additional invariance properties:
\begin{theorem}  \label{Theorem 4.7} (\cite{hochmanshmerkin2015}, Theorem 4.7) Suppose that $\mu$ generates an $S$-invariant distribution $P$. Then $P$ is supported on $\mathcal{M}^\square$ and satisfies the $S$-quasi-Palm property: for every Borel set $B\subseteq \mathcal{M}^\square$, $P(B)=1$ if and only if for every $t>0$, $P$ almost every measure $\eta$ satisfies that $\eta_{x,t}\in B$ for $\eta$ almost every $x$ such that $[x-e^{-t},x+e^{-t}]^2\subseteq [-1,1]^2$.  
\end{theorem} 

We shall refer henceforth to $S$-ergodic distributions $P$ supported on $\mathcal{M}^\square$ that satisfy the conclusion of Theorem \ref{Theorem 4.7} as EFD's (Ergodic Fractal Distributions), a term coined by Hochman in \cite{hochman2010dynamics}. The next Theorem shows that the measures we are considering in Theorem \ref{Theorem inv} and in Theorem \ref{Theorem ss} generate non-trivial EFD's. 
\begin{theorem} (\cite{hochman2010dynamics}, Section 4) \label{EFD exp}
Let $\mu\in \mathcal{P}([0,1]^2)$ be either a $T_p \times T_p$ invariant ergodic measure with $\dim \mu >0$, or a non trivial self similar measure with respect to an IFS $\Phi$ satisfying the SSC. Then $\mu$ generates a non-trivial EFD $P$. 
\end{theorem}

Next, we survey  several useful properties of EFD's.
\begin{theorem} \label{Prop of EFD}  \cite{hochman2010dynamics}
Let $P$ be an EFD and let $\pi:\mathbb{R}^2 \rightarrow \mathbb{R}$ be an orthogonal projection. Then:
\begin{enumerate}
\item $P$ almost every $\nu$ is exact dimensional. If $P$ is generated by a measure   $\mu \in \mathcal{P}([0,1]^2)$ then $P$ almost surely $\dim \nu = \dim \mu$. 

\item  For $P$ almost every $\nu$, for $\pi \nu$ almost every $x$, the conditional measure $\nu_{\pi^{-1}(x)}$ is exact dimensional and
\begin{equation*}
\dim \nu = \dim \pi \nu + \dim \nu_{\pi^{-1} (x)}
\end{equation*}

\item For $P$ almost every $\nu$ the conditional measure $\nu_{\pi^{-1}(0)}$  is well defined. 

\item Let $P_{\pi^{-1} (0)}$ be the pushforward of $P$ via the map $\nu \mapsto \nu_{\pi^{-1} (0)}$. Then $P_{\pi^{-1} (0)}$ is an EFD that is a factor of the EFD $P$. In particular, for $P$ almost every $\nu$ and $\pi \nu$ almost every $x$, $\nu_{\pi^{-1} (x)}$ generates the EFD $P_{\pi^{-1} (0)}$.
\end{enumerate}
\end{theorem}

We remark that by the statement "$P_{\pi^{-1} (0)}$ is a factor of the $P$" we mean that the dynamical system $(P_{\pi^{-1} (0)}, S)$ is a factor of the system  $(P, S)$

Finally, to an $S$-invariant distribution $P$ we associate its pure point spectrum $\Sigma(P,S)$. This set consists of all the $\alpha \in \mathbb{R}$ for which there exists  a non-zero measurable function $\phi:\mathcal{M}^{\square} \rightarrow \mathbb{C}$ such that $\phi\circ S_t = \exp(2 \pi i \alpha t)\phi$ for every $t\geq 0$, on a set of full $P$ measure. The existence of such an eigenfunction indicates that some non-trivial feature of the measures of $P$ repeats periodically under magnification by $e^\alpha$.

\subsection{Actions of diagonal endomorphisms on measures supported on lines}
We first recall the following result, about actions of diagonal endomorphisms on affine images of one dimensional measures. We denote the set of invertible real affine maps by $\text{Aff}(\mathbb{R})$.
\begin{theorem} (\cite{algom2019simultaneous}, Theorem 6.2) \label{Conjecture pert}
Let $\mu \in \mathcal{P}([0,1])$ be a probability measure, $f,g\in \text{Aff}(\mathbb{R})$ be such that $f([0,1]), g([0,1])\subseteq [0,1]$, and $m>n>1$ be integers, such that:
\begin{enumerate}
\item The measure $\mu$ generates a non-trivial EFD  $P\in \mathcal{P}(\mathcal{P}([-1,1]))$.

\item The pure point spectrum $\Sigma(P,S)$ does not contain a non-zero integer multiple of $\frac{1}{\log m}$.

\item  The measure $g\mu$ is pointwise generic under $T_n$ for an ergodic and continuous measure  $\rho$.
\end{enumerate}
Then 
\begin{equation} 
\frac{1}{N} \sum_{i=0} ^{N-1} \delta_{(T_m ^i f(x) , T_n ^i g(x))} \rightarrow \lambda\times \rho, \quad \text{ for } \mu \text{ almost every } x.
\end{equation}
\end{theorem}

The following Theorem is an almost formal consequence of the previous one:
\begin{theorem}  \label{Theorem lines}
Let $\mu \in \mathcal{P}(\ell\cap [0,1]^2)$ be a probability measure, where $\ell$ is some affine line not parallel to the principal axes, and let $m>n>1$ be integers, such that:
\begin{enumerate}
\item The measure $\mu$ generates a non-trivial EFD $P\in \mathcal{P}(\mathcal{P}([-1,1]^2))$.

\item The pure point spectrum $\Sigma(P,S)$ does not contain a non-zero integer multiple of $\frac{1}{\log m}$.

\item  The measure $P_2 \mu$ is pointwise generic under $T_n$ for an ergodic and continuous measure  $\rho$.
\end{enumerate}
Then 
\begin{equation} \label{Eq I1 pert}
\frac{1}{N} \sum_{i=0} ^{N-1} \delta_{(T_m ^i x , T_n ^i y)} \rightarrow \lambda\times \rho, \quad \text{ for } \mu \text{ almost every } (x,y).
\end{equation}
\end{theorem}

\begin{proof}
Let $\nu=P_2 \mu$, so $\nu$ is a measure on $[a,b]\subseteq [0,1]$ where $[a,b]=P_2 (\ell\cap [0,1]^2)$. Then $\nu$ satisfies conditions (1) and (2) of Theorem \ref{Conjecture pert},   which follows from e.g. (\cite{hochman2010dynamics}, Proposition 1.9). Now, let $f,g:[a,b]\rightarrow [0,1]$ be the affine maps such that $g(x)=x$, and $f(x)\in [0,1]$ is defined uniquely by requiring that $(f(x),x)\in \ell \cap [0,1]^2$ and that the map $x\mapsto (f(x),g(x))$ is surjective onto $\ell\cap [0,1]^2$. Then $g\nu=P_2 \mu$ is pointwise generic under $T_n$ for an ergodic and continuous measure  $\rho$. Finally, the measure obtained by pushing $\nu$ forward via $x\mapsto (f(x),g(x))$ is exactly the initial measure $\mu$ that we began with. Thus, the result follows from Theorem \ref{Conjecture pert}. 
\end{proof}

\section{On the proof of Theorem \ref{Theorem ss}}
\subsection{On the scaling scenery of self similar measures}
Let $\mu$ be a measure as in Theorem \ref{Theorem ss}.  We begin with the following refinement of Theorem \ref{EFD exp}. We  write $\nu_1 \sim_C \nu_2$ to indicate that the measures $\nu_1,\nu_2$ are mutually absolutely continuous with both Radon-Nikodym densities bounded by $C$, i.e. $\frac{1}{C} \leq \frac{d \nu_1}{d \nu_2} \leq C$.
\begin{Claim} (\cite{hochman2010dynamics}, Section 4.3 and Proposition 1.36) \label{Claim gen ss}
The EFD $P$ that $\mu$ generates admits a constant  $C>0$ such that: 

For $P$ almost every $\nu$  there are invertible similarity maps $f,g$ such that 
\begin{equation*}
\nu\sim_C (g\mu)|_{[-1,1]^2},\quad \mu \sim_C (h\nu)_{[0,1]^2}.
\end{equation*}
Moreover, the orthogonal parts of $f$ and $g$ belong to the group $\overline{G_\Phi}$ (the closure of $G_\Phi$), and their distribution  is given by the corresponding Haar measure.
\end{Claim}

The following Proposition is about the pure point spectrum of this EFD $P$. From here and in what follows, for $s\in \mathbb{R}$ we set 
$$ e(s):=\exp(2 \pi i s)$$.
\begin{Proposition} \label{Claim spec ss} Let $\mu$ be a self similar measure as in Theorem \ref{Theorem ss}, with a uniform contraction ratio $r$. Let $P$ be the EFD that $\mu$ generates. Let 
\begin{equation*}
k_\Phi = \min \lbrace a\in \mathbb{N}:\quad  O_1 ^a = O_2 ^{a}, \text{ for all } O_1, O_2 \in D\Phi \rbrace
\end{equation*}
where if the set on the right hand side is empty we set $k_\Phi = \infty$.

\begin{enumerate}
\item If $k_\Phi = \infty$ then $\Sigma(P,S) \subseteq \mathbb{Z}\cdot \frac{1}{\log r}$. 

\item If $k_\Phi <\infty$, let $\gamma_\Phi \in [0, 2\pi)$ be such that for any $O\in D\Phi$, $O^{k_\Phi}$ is the rotation by the angle $\gamma_\Phi$. Then
\begin{equation*}
\Sigma(P,S) \subseteq \left( \mathbb{Z}\cdot \frac{1}{\log r} \right) \bigcup \left(  \mathbb{Z}  \cdot \frac{\gamma_\Phi}{\log r \cdot k_\Phi} \right)
\end{equation*} 
\end{enumerate}
\end{Proposition}

The Proposition relies on the following Claim. While it seems standard, we could not find its exact formulation in the literature, so we give it here with full details:
\begin{Claim} \label{Lemma app}
Let $G\leq \mathbb{T}$ be a closed subgroup, and let  $T: \lbrace 0,...,p-1\rbrace^\mathbb{N} \times \mathbb{T} \rightarrow \lbrace 0,...,p-1\rbrace^\mathbb{N} \times \mathbb{T}$ be a skew-product of the form
\begin{equation*}
T(\omega, x) = ( \sigma(\omega), x+\alpha(\omega_0))
\end{equation*}
where $\sigma$ is the shift map, and $\alpha: \lbrace 0,...,p-1\rbrace \rightarrow G$. Let 
\begin{equation*}
k_0 = \min \lbrace k\in \mathbb{N}:\quad  k \alpha(i)=k \alpha(j),\quad \forall i,j\in \lbrace 0,...,p-1\rbrace \rbrace
\end{equation*}
where $k_0 = \infty$ if the set on the RHS is empty. Let $\rho$ be a $T$ invariant measure, which is the product of a (fully supported) Bernoulli measure on $\lbrace 0,...,p-1 \rbrace^\mathbb{N}$ and of the Haar measure on $G$.   Let $\Sigma$ be the pure point spectrum of the system $(\lbrace 0,...,p-1\rbrace^\mathbb{N} \times \mathbb{T}, T, \rho)$
\begin{enumerate}
\item If $k_0 =\infty$  then  $\Sigma = \emptyset$.

\item Otherwise, if there is some $i\in \lbrace 0,...,p-1\rbrace$ with $\alpha(i)\neq 0$ let $\gamma:=k_0 \cdot \alpha(i)$, or else let $\gamma=0$. Then 
\begin{equation*}
\Sigma \subseteq  \mathbb{Z}  \cdot \frac{\gamma}{k_0}
\end{equation*}
\end{enumerate}
\end{Claim}
\begin{proof}
Without the loss of generality, let $G=\mathbb{T}$. Let $f$ be an eigenfunction  with eigenvalue $\beta$. Then we have
\begin{equation} \label{Eq eigenfunction}
f(T(\omega,x))= e(\beta) f(\omega,x).
\end{equation}
Now, for a typical $\omega$, $f(\omega,\cdot)$ is a self map of $\mathbb{T}$, so we may write
\begin{equation*}
f(\omega,x) = \sum_{k\in \mathbb{Z}} g(k,\omega) e(kx),
\end{equation*}
where $g(k,\cdot): \lbrace 0,...,p-1\rbrace^\mathbb{N}\rightarrow \mathbb{C}$ (if $G\neq \mathbb{T}$ then it is a finite cyclic group and we obtain a similar representation only with a finite sum). Therefore
\begin{eqnarray*}
f(T(\omega,x)) &=& \sum_{k\in \mathbb{Z}} g(k,\sigma(\omega)) e(k(\alpha(\omega_0)+ x)) \\
& = & \sum_{k\in \mathbb{Z}} g(k,\sigma(\omega)) e(k\alpha(\omega_0))\cdot e(k x)\\
\end{eqnarray*}
Putting the last two displayed equations into \eqref{Eq eigenfunction}, we conclude that for every $k\in \mathbb{Z}$, for almost every $\omega$,
\begin{equation*}
g(k,\sigma(\omega)) e(k\alpha(\omega_0)) = e(\beta) g(k,\omega)
\end{equation*}
Since $f$ is not trivial, it follows that for some  $k\in \mathbb{Z}$,  $g(k,\cdot )\neq 0$ almost surely. Thus,
\begin{equation} \label{Eq conti}
\frac{g(k,\sigma(\omega))}{g(k,\omega)} = e(\beta - k\cdot \alpha(\omega_0)).
\end{equation}

Let $g(\omega):=g(k,\omega)$. We claim that $g(\omega)$ is constant almost surely. Let $\mu$ be the marginal of $\rho$ on the symbolic space. First, we notice that from \eqref{Eq conti} and the ergodicity of $\mu$ it follows that $|g(\cdot )|$ is constant almost surely, so we may assume this constant is $1$. It is yet another  consequence of \eqref{Eq conti} that for any $\omega,\omega' \in \lbrace 0,...,p-1\rbrace^\mathbb{N}$ up to a null set, and any $n\in \mathbb{N}$,
\begin{equation} \label{Eq shift}
\frac{g(\sigma^n(\omega))}{g(\omega)}  = \frac{g(\sigma^n(\omega'))}{g(\omega')} , \quad \text{  if } \omega|_n = \omega'|_n
\end{equation}
where $\omega|_n$ stands for the first $n$ digits of $\omega$. Now, by an application of Lusin's Theorem, for arbitrarily small $\delta>0$ there is a set $A\subseteq \lbrace 0,..., p-1 \rbrace^\mathbb{N}$ such that $\mu(A)> 1- \delta$ and $g(\omega)$ is uniformly continuous on $A$.

We now claim that for $\mu \times \mu$ almost every $(\eta,\omega)\in \lbrace 0,..,p-1\rbrace_{-\infty} ^0 \times \lbrace 0,..,p-1\rbrace^\mathbb{N}$, there is a subset $D_{(\eta,\omega)} \subseteq \mathbb{N}$  of density $\mu(A)$ such that for every $n\in D_{(\eta,\omega)}$
\begin{equation*}
\eta_{-n}....\eta_{-1}\eta_0 \omega \in  A
\end{equation*}
where by $\eta_{-n}....\eta_{-1}\eta_0 \omega$ we mean the concatenation of the finite word $\eta_{-n}....\eta_{0}$ with the (infinite) string $\omega$. Indeed, since $\mu$ is a Bernoulli measure, we may write $\mu=p^\mathbb{N}$ for some non-trivial probability vector. Consider the Rokhlin extension of $(\lbrace 0,..,p-1\rbrace^\mathbb{N}, \sigma, \mu)$, which is given by $(\lbrace 0,..,p-1\rbrace^\mathbb{Z}, \sigma, p^\mathbb{Z})$,  since $\mu$ it is Bernoulli.  Then the claim is a consequence of the ergodic theorem applied to the indicator function of $\lbrace 0,...,p-1 \rbrace_{-\infty} ^0 \times A$, and the map $\sigma^{-1}$.

 Therefore, by Fubini's Theorem, for $\mu$ almost every $\eta\in \lbrace 0,..,p-1\rbrace_{-\infty} ^0 $ there is a set $R_\eta \subseteq \lbrace 0,..,p-1\rbrace^\mathbb{N}$ of full $\mu$ measure such that for every $\omega \in R_\eta$,
\begin{equation} \label{Eq sim}
\eta_{-n}....\eta_{-1}\eta_0 \omega \in  A \quad \text{ for every } n\in D_{(\eta,\omega)}, \text{ which has density } \mu(A).
\end{equation}

Now, fix a typical $\eta$ and  let $\omega, \omega'\in R_\eta$. Assuming $\mu(A)$ is very large, we find that the set 
$$D_{\omega,\omega',\eta}=D_{(\eta,\omega)}\cap D_{(\eta,\omega')} \subseteq \mathbb{N}$$  
where \eqref{Eq sim}  happens simultaneously for $\omega$ and $\omega'$, has positive (lower) density. For any $n\in D_{\omega,\omega',\eta}$ we have that both $\eta_{-n}....\eta_{0}\omega$ and  $\eta_{-n}....\eta_{0} \omega'$ are in $A$. Since $g$ is uniformly continuous on $A$, making use of \eqref{Eq shift} and that $|g(\cdot )|$ is constant, we see that for any such $n$
\begin{equation*}
|g(\eta_{-n}....\eta_{-1}\omega) - g(\eta_{-n}....\eta_{0}\omega')| = o(n) \Longrightarrow |g(\omega) - g(\omega')| = o(n).
\end{equation*}
Since $D_{\omega,\omega',\eta}$  has lower positive density it is infinite. It follows that $g(\omega)=g(\omega')$ for all $\omega,\omega' \in R_\eta$.  Since $\mu (R_\eta)=1$, we conclude that $g$ is almost surely constant.

Finally, by \eqref{Eq conti}, $\beta - k\alpha(\omega_0)=0 \mod 1$ for every $\omega_0$. This is sufficient for the Claim.
 \end{proof}
 
\textbf{Proof of Proposition \ref{Claim spec ss}} We first (partially) recall the construction of $P$, carried out in (\cite{hochman2010dynamics}, Section 4.3). Let 
\begin{equation*}
\Phi = \lbrace \phi_k(x)=r \cdot O_k\cdot x+t_k \rbrace_{k=0} ^{q-1}
\end{equation*}
be the underlying IFS, where $0<r<1$, $O_k$ are $2\times 2$ rotation matrices for every $k$, and $t_k\in \mathbb{R}^2$. Let $X$ be the attractor of $\Phi$. By (\cite{hochman2010dynamics}, Section 4.3) we may assume that there is an open set $A$ such that $X\subseteq A$ and $\phi_k (A) \subset A$ are pairwise disjoint for every $k$. In addition, let $F_\Phi:\lbrace0,..,q-1\rbrace^\mathbb{N}\rightarrow X$ be the usual (continuous and onto) coding map 
\begin{equation*}
F_\Phi(\omega) = \lim_{k\rightarrow \infty} \phi_{\omega_1}  \circ \phi_{\omega_2}\circ ...\circ \phi_{\omega_n} (0).
\end{equation*}
Then it is well known that our self similar measure $\mu$ admits a unique $F_\Phi
$-lift $\tilde{\mu}\in \mathcal{P}(\lbrace0,...,q-1\rbrace^\mathbb{N})$ where $\tilde{\mu}$ is a stationary Bernoulli measure \cite{hochman2010dynamics}. We also recall that $G_\Phi$ is the group generated by the $O_k$'s. Let $\rho$ denote the Haar measure on $\overline{G_\Phi}$. Notice that $\overline{G_\Phi}$ is either a finite cyclic group or $SO(\mathbb{R}^2)$, which is isomorphic to $\mathbb{T}$.

Consider the distribution $Q$ on $ \mathbb{R}^2 \times \overline{G_\Phi}$ defined as follows: We choose $y\sim \mu$ and $U \sim \rho$ independently, and consider $(Uy, U)$. We now define a map $M:\supp(Q)\rightarrow \supp(Q)$. Let $(y,V)\in \supp(Q)$ be such that  $V^{-1} (y) \in \supp (\mu)$. We then define $M(y,V) = (y',V')$ by choosing the unique index such that $V^{-1}(y)\in \phi_i (A)$, and then taking
\begin{equation*}
 y' = V\phi_i ^{-1} V^{-1} y,\quad V'=O_i V
\end{equation*}
One may verify that $M$ is well defined $Q$ almost surely, and that  $Q$ is $M$ invariant.

Let $G:= \overline{G_\Phi}$.  Notice that the system $(\supp(Q),Q,M)$ is a factor of the system \newline $(\lbrace 0,...,q-1\rbrace^\mathbb{N} \times G, T, \tilde{\mu} \times \rho)$ where $T$ is a skew-product as in Lemma \ref{Lemma app}, with $\alpha(i)$ arising from $O_i$, the rotation part of $\phi_i$.  Indeed,    a factor map  is given by
\begin{equation*}
(\omega, t)\mapsto (e(t) F_\Phi(\omega),\quad  e(t)),\quad \text{where } e(t) \text{ is the rotation by the angle } t.
\end{equation*}

Finally, the EFD $P$ that $\mu$  generates arises as follows: it is a factor of the suspension by a function of constant height $\log r$ (recall that $r$ is the uniform contraction ratio associated with this IFS) of a factor of $(\supp(Q),Q,M)$. The latter system is a factor of $(\lbrace 0,...,q-1\rbrace^\mathbb{N} \times G, T, \tilde{\mu} \times \rho)$. By the preceding discussion and Claim \ref{Lemma app}, the pure  point spectrum of $(P,S)$ satisfies the claimed conditions, where the $\log r$ factor arises from taking the suspension  (\cite{hochman2010geometric}, Section 3.5).

\subsection{Proof of Theorem \ref{Theorem ss}}
We now fix a measure $\mu$ as in Theorem \ref{Theorem ss}, and recall that $P$ is the EFD that $\mu$ generates. The following Claim is the key to the proof of Theorem \ref{Theorem ss}.

\begin{Claim}  \label{Claim dis for ss}
 There exists an orthogonal projection $\pi:\mathbb{R}^2 \rightarrow \mathbb{R}$ with $\pi \notin \lbrace P_1,P_2 \rbrace$ such that:
\begin{enumerate}

\item There is a non-trivial EFD $Q$ such that almost every conditional measure $\mu_{\pi^{-1} (x)}$ generates $Q$. Moreover, $(Q,S)$ is a factor of $(P,S)$, where $S$ is the scaling flow on $\mathcal{M}^\square$.

\item For almost every conditional measure $\mu_{\pi^{-1} (x)}$,  $P_2 \mu_{\pi^{-1} (x)}$ is pointwise generic under $T_n$ for $\lambda$.
\end{enumerate} 
\end{Claim}

\textbf{Proof of Theorem \ref{Theorem ss} assuming Claim \ref{Claim dis for ss}} Let $\pi$ be the orthogonal projection from Claim \ref{Claim dis for ss}. Then  we see via Claim \ref{Claim dis for ss} that typical conditional measures with respect to $\pi$ satisfy the requirements of  Theorem \ref{Theorem lines}. Indeed, since $(Q,S)$ is a factor of $(P,S)$,   the spectrum of the EFD $Q$ from Claim \ref{Claim dis for ss} part (2) satisfies $ \Sigma(Q,S) \subseteq \Sigma (P,S).$
Therefore $\Sigma(Q,S)$ does not contain a non-zero integer multiple of $\frac{1}{\log m}$ by Claim \ref{Claim spec ss} and our assumptions on $m,r,\gamma_\Phi,k_\phi$.
 Thus, the Theorem follows via an application of  Theorem \ref{Theorem lines} to the typical conditional measures $\mu_{\pi^{-1} (x)}$, and since 
 $$\mu = \int \mu_{\pi^{-1} (x)} d \pi \mu (x). $$ 
 
 \textbf{Proof of Claim \ref{Claim dis for ss}} We begin by producing a projection $\pi \notin \lbrace P_1,P_2 \rbrace$ such that typical conditional measures $\mu_{\pi^{-1} (x)}$ have positive dimension.  Assume first that $|G_\Phi|<\infty$. Then by dimension conservation (\cite{furstenberg2008ergodic}, \cite{Falconer2014Jin}), for every projection $\pi:\mathbb{R}^2 \rightarrow \mathbb{R}$ we have
\begin{equation} \label{Eq dim conservation}
\dim \mu = \dim \pi \mu + \dim \mu_{\pi^{-1} (x)},\quad \text{ almost surely.}
\end{equation}
If $\dim \mu \leq 1$ then by Theorem \ref{Theorem farkas}  there exists an orthogonal projection $\pi:\mathbb{R}^2 \rightarrow \mathbb{R}$ such that $\dim \pi \mu < \dim \mu \leq 1$ and  $\pi \notin \lbrace P_1,P_2 \rbrace$. By \eqref{Eq dim conservation}, we obtain that typical conditional measures with respect to $\pi$ have positive dimension. If $\dim \mu>1$ then, via \eqref{Eq dim conservation}, every non-principal projection will work.

Assume now that $|G_\Phi|=\infty$ and that $\dim \mu >1$. Let $P$ be the EFD generated by $\mu$. By Theorem \ref{Prop of EFD}, $P$ almost every $\nu$ is exact dimensional with $\dim \nu = \dim \mu$. Also,  for every projection $\pi:\mathbb{R}^2 \rightarrow \mathbb{R}$, for $P$ almost every $\nu$
\begin{equation} \label{Eq dim conservation2}
\dim \nu = \dim \pi \nu + \dim \nu_{\pi^{-1} (x)},\quad \text{ almost surely.}
\end{equation}
Let $\nu$ be such a $P$ typical measure. Then, by Claim \ref{Claim gen ss}, there is some $C>0$ and an element $U\in SO(\mathbb{R}^2)$ such that $U\mu \sim_C \nu|_{B}$, where $B$ is some ball. In particular, typical $\pi$-conditional measures of these measure are also equivalent. Since typically 
$$(U \mu)_{\pi^{-1}(x)} = U(\mu_{{(U\circ \pi)}^{-1} (x)})$$
we obtain via \eqref{Eq dim conservation2} and exact-dimensionality of these conditional measures that
\begin{equation*}
\dim \mu = \dim (\pi \circ U)\mu + \dim \mu_{{(U\circ \pi)}^{-1} (x)},\quad \text{ almost surely.}
\end{equation*} 
Since $\dim \mu >1$, we may take $\pi\circ U$ as our desired projection. It is clear that we may assume $\pi \circ U \neq P_1,P_2$,  since $U$ is distributed according to the Haar (Lebesgue) measure.

So far we have produced a projection $\pi$ such that typical conditional measures $\mu_{\pi^{-1} (x)}$ have positive dimension. Next,  let $P_{\pi^{-1}(0)}$ denote the push-forward of $P$ via the map $\mu \mapsto \mu_{\pi^{-1}(0)}$. By Theorem \ref{Prop of EFD} the distribution $P_{\pi^{-1}(0)}$ is an EFD that is a factor of $P$, and for $P$ almost almost every $\nu$, almost every conditional measure $\nu_{\pi^{-1}(x)}$  generates $P_{\pi^{-1}(0)}$.

If $|G_\Phi|<\infty$  then by Claim \ref{Claim gen ss} there is some $C>0$ such that  $\mu \ll_C (h \nu)|_B$ (i.e. the density is bounded by $C$), where $\nu$ is $P$ typical, $B$ is some ball, and  $h$ is a \textit{homothety}. It follows that almost every conditional measure of $\mu$ with respect to $\pi$ is absolutely continuous with respect to the corresponding conditional measures of $(h \nu)|_B$. So, the $\pi$-conditionals of $\mu$ generate the same EFD that the $\pi$-conditionals of $(h \nu)|_B$ generate, which in turn is the same EFD that the $\pi$-conditionals of $\nu$ generate, which is $P_{\pi^{-1}(0)}$ (here we use that $h$ is a homothety and apply (\cite{hochman2010dynamics}, Proposition 1.9)).  Thus, if $|G_\Phi|<\infty$ we take $Q=P_{\pi^{-1}(0)}$ which is a factor of $P$, and is non-trivial since it is generated by typical conditional measures $\mu_{\pi^{-1} (x)}$, and they have positive dimension (\cite{hochman2010dynamics}, Proposition 1.19).

If $|G_\Phi|=\infty$ then (up to a restriction to some ball) $U\mu \sim_C \nu$, where $\nu$ is the same measure we worked with when we produced $\pi$. It again follows that almost every conditional measure of $\mu$ with respect to $U\circ \pi$ is absolutely continuous with respect to the corresponding conditional measures of $\nu$ with respect to $\pi$, up to pushing forward via $U$. Therefore, $\pi\circ U$ typical conditional measures generate the  EFD $U^{-1} P_{\pi^{-1}(0)}$ via (\cite{hochman2010dynamics}, Proposition 1.9). Recalling that in this situation our projection was defined as the composition of $\pi$ and $U$, this yields the desired result. Thus, if $|G_\Phi|<\infty$ we take $Q=U^{-1} P_{\pi^{-1}(0)}$. Notice that this is a factor of $P_{\pi^{-1}(0)}$ by (\cite{hochman2010dynamics}, Section 7.1), and thus also a factor of $P$. Also, $Q$ is non-trivial since it is generated by typical conditional measures, that have positive dimension. This concludes the first part of the Claim.

For the second part of the Claim, we recall almost surely the measure $\mu_{\pi^{-1} (x)}$ generates   $Q$. Moreover, $Q$ is non trivial. Since $Q$ is a factor of $P$,  $\Sigma(Q,S) \subseteq \Sigma(P,S)$, so $\Sigma(Q,S)$ does not contain any integer multiple of $\frac{1}{\log n}$.   It follows from (\cite{hochman2010dynamics}, Proposition 1.9) that $P_2 \mu_{\pi^{-1} (x)}$ also generates a non-trivial EFD such that its pure point spectrum does not contain any integer multiple of $\frac{1}{\log n}$. It now follows from the main result of \cite{hochmanshmerkin2015} that $P_2 \mu_{\pi^{-1} (x)}$ is pointwise generic under $T_n$ for $\lambda$.

\section{On the proof of Theorem \ref{Theorem inv}}
Let $\mu$ be a $T_p \times T_p$ invariant measure as in Theorem \ref{Theorem inv}. In particular, we assume that there exists a projection $\pi:\mathbb{R}^2 \rightarrow \mathbb{R}$ with $\pi\neq P_1,P_2$ such that $\dim \pi \mu < \dim \mu.$ We will (occasionally) make use of the standard identification of the map $T_p \times T_p$ on $[0,1]^2$ with the shift map on \newline $(\lbrace0,...,p-1\rbrace \times \lbrace0,...,p-1 \rbrace) ^\mathbb{N}$, given by the base $p$ expansion (in dimension $2$). This is defined uniquely off a countable set of horizontal and vertical lines, and by combining the condition $\dim \pi \mu < \dim \mu$ with \eqref{Eq dim conservation}, that holds for $\mu$ by \cite{furstenberg2008ergodic}, one may verify that it is an isomorphism almost surely.  
\subsection{Disintegration of conditional measures}
In this section we find a disintegration of typical conditional measures  $\mu_{\pi^{-1}(x)}$ into conditional measures that almost surely satisfy the conditions of Theorem \ref{Theorem lines}. Recall that by Theorem \ref{EFD exp}  $\mu$ generates a non-trivial EFD $P$. Given our orthogonal projection $\pi:\mathbb{R}^2 \rightarrow \mathbb{R}$,  let $P_{\pi^{-1}(0)}$ be the EFD as in Theorem \ref{Prop of EFD} part (3).
\begin{Claim} \label{Claim dis}
There exists a family of distributions $\lbrace P_x \rbrace_{x\in \supp (\pi \mu)} \subseteq \mathcal{P}(\mathcal{P}([0,1]^2))$ such that:
\begin{enumerate}
\item For $\pi \mu$ almost every $x$, $P_x$ almost every measure $\nu$ generates the EFD  $P_{\pi^{-1}(0)}$.

\item For $\pi \mu$ almost every $x$ we may disintegrate the conditional measure $\mu_{\pi^{-1} (x)}$ as
\begin{equation*}
\mu_{\pi^{-1} (x)} = \int \nu  dP_x(\nu)
\end{equation*}

\end{enumerate}
\end{Claim}

\begin{proof}
We first recall the explicit construction of the EFD that $\mu$ generates (\cite{hochman2010geometric}, Section 3.2). We identify our measure $\mu$ with an ergodic shift invariant measure $\mu$ on 
$$\lbrace 0,...,p-1\rbrace^\mathbb{N}\times \lbrace 0,...,p-1\rbrace^\mathbb{N}.$$
Its natural extension is a shift invariant measure $\tilde{\mu}$ on 
$$\lbrace 0,...,p-1\rbrace^\mathbb{Z}\times \lbrace 0,...,p-1\rbrace^\mathbb{Z},$$
 that projects to $\mu$ on the positive coordinates $\lbrace (i,j): i,j\geq 0 \rbrace$. Then we may disintegrate the measure $\mu$ according to its past, i.e. 
\begin{equation} \label{Eq well defined}
\mu = \int \nu_\omega d\tilde{\mu}(\omega),
\end{equation} 
where  $\nu_\omega$ is the distribution of $\sum_{i=1} ^\infty p^{-i} \omega_i$ given $(\omega_i)_{i\leq 0}$, see (\cite{hochman2010geometric}, Theorem 3.1). We remark that $\tilde{\mu}$ serves as a discrete analogue of $P$, and indeed $P$ arises by considering a certain suspension flow related to $\tilde{\mu}$, a fact that we shall make use of in this proof.

Now, recall that we are given an orthogonal projection $\pi$. Suppose without the loss of generality that $\pi([0,1]^2)\subseteq [-1,1]$. We disintegrate the measure $\mu$ as follows: consider the  measure valued map $[-1,1]^2 \times \supp(\tilde{\mu}) \rightarrow \mathcal{P}([-1,1]^2)$ defined (almost surely) by
\begin{equation*}
(x,\nu_\omega) \mapsto (\nu_\omega)_{\pi^{-1}(x)},
\end{equation*}
where we first draw $\nu_\omega$ according to $\tilde{\mu}$ and then draw $x$ according to $\pi \nu$. Now, fix a typical $x$, and let $P_x$ denote the distribution  of $ (\nu_\omega)_{\pi^{-1}(x)}$ given $x$. We remark that by \eqref{Eq well defined} the distribution of $x$ is given by $\pi \mu$.

Next, as in \cite{hochman2010dynamics}, let $\text{cent}_0(\tilde{\mu})$ be the distribution defined by first choosing a measure $\nu_\omega$ according to $\tilde{\mu}$, then drawing a $\nu_\omega$ typical point $x$, and then looking at the translated measure $ \nu_\omega ^x$ (so that $0 \in \supp(\nu_\omega ^x)$). By (\cite{hochman2010geometric}, Proposition 3.6) $P$ arises as
\begin{equation} \label{Eq P }
P = \int_0 ^1 S_{t \log p}  \left( \text{cent}_0(\tilde{\mu}) \right) dt.
\end{equation}
Also, recall that $P_{\pi^{-1} (0)}$ is defined by first drawing $\nu$ according to $P$ and then looking at the conditional measure $(\nu)_{\pi^{-1} (0)}$.  This is the same distribution as the one we get by considering measures of the form $(S_t  \nu^y _\omega)_{\pi^{-1}(0)}$, where we first draw $\nu_\omega$ according to $\tilde{\mu}$, then $y$ according to $\nu_\omega$, then $t$ according to Lebesgue, and finally condition on $\pi^{-1}(0)$. It follows from  (\cite{hochman2010dynamics}, Claim 6.12 and the preceding Lemmas)    that this is the same distribution    as the one on measures of the form $S_t \left( (\nu_\omega)_{\pi^{-1} (x)} \right)^y$, where we first draw $x$ according to $\pi \mu$, and then the conditional measure $(\nu_\omega)_{\pi^{-1}(x)}$ according to $P_x$,  translate by random $y$ in this fiber, and scale by a Lebesgue typical $t$.

Now, since $P_{\pi^{-1} (0)}$ is an EFD, it is generated by $P_{\pi^{-1} (0)}$ almost every $\nu$. We have just shown that these measures have the same asymptotic scaling sceneries as $(\nu_\omega)_{\pi^{-1}(x)}$, where we first draw $x$ according to $\pi \mu$ and then the conditional measure $(\nu_\omega)_{\pi^{-1}(x)}$ according to $P_x$. So, for $\pi \mu$ almost every $x$, $P_x$  almost every $\nu$ generates $P_{\pi^{-1} (0)}$. This concludes the proof of the first part of the Claim.

For the second part, we use a Fubini argument. By \eqref{Eq well defined} and by the definition of $P_x$,
\begin{eqnarray*}
\mu &=&  \int \nu_\omega d\tilde{\mu}(\omega)\\
&=& \int \int (\nu_\omega)_{\pi^{-1} (x)} d\pi \nu_\omega (x) d \tilde{\mu}(\omega) \\
&=& \int \int  (\nu_\omega)_{\pi^{-1} (x)} d P _x (\omega) d(\int \pi \nu_\omega  d\tilde{\mu}(\omega) )(x) \\
&=& \int \int (\nu_\omega)_{\pi^{-1} (x)} d P_x (\omega) d\pi \mu (x)
\end{eqnarray*} 
so that almost surely, 
\begin{equation*}
\mu_{\pi^{-1} (x)} = \int \nu d P_x (\nu). 
\end{equation*}
\end{proof}

The following Claim is where the assumption on $\pi$ from \eqref{Equation condition} is used:
\begin{Claim} \label{Claim proj inv}
For our projection $\pi:\mathbb{R}^2 \rightarrow \mathbb{R}$ the EFD $P_{\pi^{-1}(0)}$ is not trivial.
\end{Claim}
\begin{proof}
It suffices to show that $P_{\pi^{-1}(0)}$ typical measures have positive dimension. Recall that we are assuming $\dim \pi \mu < \dim \mu$ for our original measure $\mu$. Now, by Theorem \ref{Prop of EFD} parts (1) and (2) we find that
\begin{equation} \label{Eq dim cons FD}
\dim \pi \nu + \dim \nu_{\pi^{-1} (x)} = \dim \nu = \dim \mu,
\end{equation}
for $P$ almost every $\nu$ and $\nu$ almost every $x$. Now, by (\cite{hochman2010dynamics}, Theorem 1.23), for $P$ almost every $\nu$,
\begin{equation} \label{dim drop}
 \dim \pi \mu \geq \int \dim (\pi \eta) dP(\eta) = \dim \pi \nu
\end{equation}
where the last equality is due to $P$ being $S$-ergodic. Thus, a dimension drop for $\pi \mu$ implies a dimension drop of $\pi \nu$ for almost every $\nu$. Therefore,  \eqref{Eq dim cons FD} implies that $P$ almost surely, the conditional measures according to $\pi$ almost surely have positive dimension. Since the dimension of $P_{\pi^{-1}(0)}$ typical measures is equal to the dimension of typical $\pi$-conditional measures of $P$ typical measures, we are done.
\end{proof}

\subsection{Proof of Theorem \ref{Theorem inv}}

\textbf{Proof of Theorem \ref{Theorem inv} under a spectral assumption} Let us first assume that the measure $\mu$ generates an EFD $P$ such that for every integer $k\neq 0$, $\frac{k}{\log m} \notin \Sigma(P,S)$. Recall the family of distributions $\lbrace P_x \rbrace_{x\in \supp (\pi \mu)}$ from  Claim \ref{Claim dis}. Then for $\pi \mu$ almost every $x$, $P_x$ almost every $\nu$  generates the EFD $P_{\pi^{-1}(0)}$.  Now, the  EFD $P_{\pi^{-1}(0)}$ is a factor of $P$, so its pure point spectrum is contained in that of $P$. Thus,  for every integer $k\neq 0$, $\frac{k}{\log m}\notin \Sigma(P_{\pi^{-1}(0)},S)$.  Also, by Claim \ref{Claim proj inv}, the EFD $P_{\pi^{-1}(0)}$ is not trivial.  Finally, the measure $P_2 \mu$ is  pointwise generic under $T_n$ for either $P_2 \mu$ (if $n=p$ by the ergodic Theorem) or for $\lambda$ (if $n\not \sim p$ by\footnote{Notice that $\dim P_2 \mu >0$.  Indeed,  $\dim \mu_{\pi^{-1} (x)} =\dim \mu-\dim \pi \mu$ almost surely, and $P_2 \mu = \int P_2 \mu_{\pi^{-1} (x)} d\pi \mu(x)$ so $$\dim P_2 \mu \geq \text{ess-inf}_{x\sim \pi \mu} \dim P_2 \mu_{\pi^{-1} (x)}=\dim \mu-\dim \pi \mu>0$$} (\cite{hochmanshmerkin2015}, Theorem 1.10)). Therefore, the same is true for $P_2 \mu_{\pi^{-1} (x)}$ almost surely. Thus, the same is true for $P_2 (\nu)$ for $\pi \mu$ almost every $x$ and  $P_x$ almost every $\nu$.

Therefore, for $\pi \mu$ almost every $x$,  $P_x$ almost every measure $\nu$   meets the conditions of Theorem \ref{Theorem lines}. Thus,  the measure $\nu$ is pointwise generic under $T_m \times T_n$ for either $\lambda \times P_2 \mu$ or $\lambda \times \lambda$ (depending on the nature of $n$). Finally, by Claim \ref{Claim dis} part (2),
\begin{equation} \label{Eq dis for ori}
\mu = \int \mu_{\pi^{-1} (x)} d \pi \mu (x) = \int \left( \int \nu  dP_x(\nu) \right) d \pi \mu (x)
\end{equation}
so $\mu$ itself is pointwise generic under $T_m \times T_n$ for either $\lambda \times P_2 \mu$ or $\lambda \times \lambda$, depending on the nature of $n$.

\textbf{Proof of Theorem \ref{Theorem inv} without a spectral assumption} We remain with the case when the spectral condition on $P$ is not met: there exists some $k\in \mathbb{Z}\setminus \lbrace 0 \rbrace$ such that $\frac{k}{\log m}\in \Sigma(P,S)$. Then   $P$ is not $S_{\log m}$ ergodic by  (\cite{hochmanshmerkin2015}, Proposition 4.1). Thus, $P_{\pi^{-1} (0)}$ is not necessarily $S_{\log m}$ ergodic. By the previous case's proof, we may assume that indeed it is not ergodic.

Let $Q:= P_{\pi^{-1} (0)}$ and let $Q = \int Q_\eta d Q (\eta)$ denote the ergodic decomposition of $Q$ with respect to the action of $S_{\log m}$. So, a $Q$ typical measure $\eta$ generates $Q$, and $S_{\log m}$ generates its $S_{\log m}$ ergodic component $Q_\eta$ (this means that for $\eta$ typical $y$, the sequence $\lbrace \eta_{y, k\log m}\rbrace_k$ equidistributes for $Q_\eta$), which follows from from (\cite{hochmanshmerkin2015}, Lemma 4.2). As in the previous case, we aim to prove that:

\begin{Claim} \label{Claim spectral}
For $\pi \mu$ almost every $x$, $P_x$ almost every $\nu$ is pointwise generic under $T_m \times T_n$ for $\lambda \times P_2 \mu$ if $n=p$ and for $\lambda \times \lambda$ if $m\neq n$ and $n\not \sim p$.
\end{Claim}
Assuming Claim \ref{Claim spectral} is true, the result follows from  \eqref{Eq dis for ori}. The proof of Claim \ref{Claim spectral} is  similar to the proof of (\cite{algom2019simultaneous}, Theorem 1.1), given in (\cite{algom2019simultaneous}, Section 5), with minor modifications.   Thus, we only sketch the details:
 By Claim \ref{Claim proj inv}  the dimension of $Q$ typical measures is positive almost surely, and is constant almost surely by Theorem \ref{Prop of EFD}. Let $\delta>0$ be this almost surely value.   Now, a $Q$ typical measure $\kappa$ $S_{\log m}$ generates an ergodic component $Q_{\eta(\kappa)}$. Thus, for $\pi \mu$ typical $x$, $P_x$ typical $\nu$, for $\nu$ typical $y$, and for $\lambda$ typical $t$,  the sequence $\lbrace \nu_{y, k\log m+t}\rbrace_k$ also equidistributes for a typical ergodic component $Q_{\eta(\nu,t)}$ (this follows from the proof of the first part of Claim \ref{Claim dis}). Moreover, almost every ergodic component of $Q$ arises this way.

By an analogue of (\cite{hochmanshmerkin2015}, Lemma 8.3) for every measure $\tau\in \mathcal{P}([0,1])$ such that $\dim \tau \geq 1- \delta$,  we have for almost every $t$,
\begin{equation} \label{Equation T}
1=  \dim \tau * \rho,\quad \text{ for } Q_{\eta(\nu,t)} \text{ almost every } \rho.
\end{equation} 
For this to work, we notice that integrating $\dim \tau * \rho$ against 
$$dQ_{\eta(\nu,t)} (\rho)dt dP_x (\nu) d\pi \mu (x)$$
 is the same as integrating it against $dQ_\eta (\rho) dQ(\eta)= dQ$, which is an EFD and $Q$ typical measures have dimension $\delta$. Therefore, the result follows as in (\cite{hochmanshmerkin2015}, Lemma 5.8). Now, let 
 \begin{equation*}
\Theta = \left\{
  \beta^{*k} \;\middle|\;
  \begin{aligned}
  & \beta \sim \sum_{i=1} ^\infty \frac{X_i}{m^{qi}} \text{ where  } q\in \mathbb{N}  \text{ and } \lbrace X_k \rbrace  \text{ forms an IID sequence such that }  \\
  & \mathbb{P} (X_1=0)=\frac{1}{3}, \mathbb{P} (X_1=1)=\frac{2}{3}, \text{ for any } k\in \mathbb{N} \text{ satisfying } \dim \beta^{*k} > 1-\delta
  \end{aligned}
\right\}
\end{equation*}
  where $\beta^{*k}$ stands for the self-convolution of $\beta$ with itself $k$ times. The set $\Theta$ is countable and not empty by (\cite{Elon1999conv}, Theorem 1.1).

Finally, for a $\pi \mu$ typical $x$,  fix a $P_x$ typical measure $\nu$  such that: for some $t$, $Q_{\eta(\nu,t)}$ is a typical ergodic component with respect to \eqref{Equation T}, for every $\tau \in \Theta$.  Let $z$ be a $\nu$ typical point, and  let $\alpha$ be a measure such that $z$ equidistributes for it sub-sequentially under $T_m\times T_n$. We want to show that $\alpha = \lambda \times P_2 \mu$ if $n=p$ or $\alpha = \lambda \times \lambda$ if $m>n\not \sim p$. We assume without the loss of generality that $n= p$. By the decomposition in Claim \ref{Claim dis} part (2), we may assume $P_2 \alpha = P_2 \mu$, since  $P_2\mu$ is pointwise generic for $P_2 \mu$ by the ergodic Theorem, and therefore a.s.  $P_2 \mu_{\pi^{-1}(x)}$ is pointwise generic for $P_2 \mu$. Since for $\nu$ almost every $y$, $\lbrace \nu_{y, k\log m+t}\rbrace_k$ equidistributes for $Q_{\eta(\nu,t)}$,  we obtain an integral representation of $P_2$-conditional measures of $\alpha$ as in\footnote{While Claim 4.1 in \cite{algom2019simultaneous} requires that $\lbrace \nu_{y, k\log m+t}\rbrace_k$ equidistributes for a typical ergodic component $Q_\eta$ for $t=0$, an almost identical proof also  yields the case $t\neq0$.} (\cite{algom2019simultaneous}, Claim 4.1) with respect to $Q_{\eta(\nu,t)}$.   By the choice of $Q_{\eta(\nu,t)}$,  for every $\tau \in \Theta$ and   for $P_2 \mu$ almost every $y$, $\dim \tau*(P_1\alpha_{P_2 ^{-1}(y)}) =1$ (by the proof of Theorem 1.2 in \cite{algom2019simultaneous}). Finally, by\footnote{Notice that formally, Claim 3.4 in \cite{algom2019simultaneous} requires $\dim \tau*(P_1\alpha_{P_2 ^{-1}(y)}) =1$ to hold a.s. for every $\tau \in \mathcal{P}([0,1])$ with $\dim \tau \geq 1- \delta$. However, for the proof to work we only really need this to hold for every measure in the countable family $\Theta$.} (\cite{algom2019simultaneous}, Claim 3.4) we find that  $\alpha = \lambda \times P_2\mu$. This concludes the proof of Claim \ref{Claim spectral}.

\bibliography{bib}{}
\bibliographystyle{plain}

\end{document}